\documentclass{amsart}

\usepackage{tikz}
\usetikzlibrary{matrix}
\usepackage[utf8]{inputenc}
\usepackage{graphicx,subcaption,amssymb,rotating}
\usetikzlibrary{matrix}
\usetikzlibrary{arrows}
\usepackage{hyperref}

\newtheorem{theorem}{Theorem}[section]
\newtheorem{corollary}{Corollary}[theorem]
\newtheorem{lemma}[theorem]{Lemma}

\theoremstyle{definition}

\theoremstyle{remark}
\newtheorem{remark}[theorem]{Remark}

\numberwithin{equation}{section}

\begin{document}

\title[The NEF cone of Hilbert scheme in the Grassmannian]{The NEF cone of the Hilbert Scheme of hypersurfaces in the Grassmannian}

\author{See-Hak Seong}
\address{Department of Mathematics,Statistics and CS, University of Illinois at Chicago, Chicago, Illinois 60607}
\email{sseong3@uic.edu}

\begin{abstract}
We show that when $d \geq 3$ and $m>2$, the Nef cone of the Hilbert scheme $Hilb_{P_{d,m}(T)}(G(k,n))$ is a cone spanned by 6 classes in general case, where $P_{d,m}(T)=\binom{T+m}{m}-\binom{T+m-d}{m}$.
\end{abstract}

\maketitle

\section*{Contents}
~\\
1. Introduction \hfill 1 \\
2. Preliminaries \hfill 2\\
3. Nef cone of $Hilb_{P_{d,m}(T)}(G(k,n))$ for $m> 2,d\geq 3$ \hfill 4\\
References

\section{Introduction}

Let $P_{d,m}(T)=\binom{T+m}{m}-\binom{T+m-d}{m}$. This is the Hilbert polynomial of a hypersurface of degree $d$ in $\mathbb{P}^m$. Let $F(k_1,k_2;n)$ denote the two step flag variety that parameterizes flags $W_1 \subset W_2 \subset V$, where $\dim(W_i)=k_i$ and $\dim(V)=n$. Let $\mathcal{S}$ denote the tautological bundle over a 2-step flag variety $F(k_1,k_2;n)$ is defined by taking each fiber as a projectivized quotient space $\pi^{-1}(W_1,W_2):= \mathbb{P}(W_2/W_1).$

 In \cite{3}, we showed that the Hilbert scheme $Hilb_{P_{d,m}(T)}(G(k,n))$ is a projective bundle over disjoint union of flag varieties.

\begin{theorem}
For $d \geq 3, 1<k<n-1$ and $m\geq 2$, let's define $P_{d,m}(T) = \binom{T+m}{m}-\binom{T+m-d}{m}$. Let $\mathcal{S}$ be a tautological bundle over the 2-step flag variety. Then the $Hilb_{P_{d,m}(T)}(G(k,n))$  is the projective bundle $\mathbb{P}(Sym^d \mathcal{S}^*)$ over the following flag varieties
\begin{center}
$\begin{array}{ll}
2< m \leq k \text{ case} & : F(k-1,k+m;n) \sqcup F(k-m,k+1;n) \\
k < m \leq n-k \text{ case} & : F(k-1,k+m;n) \\
n-k< m \text{ case} & : \emptyset \\
\end{array}$
\end{center}
where $\mathcal{S}$ is the tautological bundle over a 2-step flag variety $F(k_1,k_2;n)$.
\end{theorem}

And we also show that every closed subscheme of Hilbert polynomial $P_{d,m}(T)$ is a degree $d$ hypersurface of $m$-dimensional linear projective space. 

In this paper, we computed the Nef cone of the Hilbert scheme by using two divisors $D_X$ and $D_Y$ of the Hilbert scheme $Hilb_{P_{d,m}(T)}(G(k,n))$, which are defined as follows. Let $X:= \Sigma_m(F_{\bullet})$ and $Y:= \Sigma_{1,\cdots,1}(F'_{\bullet})$ be $m$-codimensional special Schubert subvarieties of $G(k,n)$, where $F_{\bullet}, F'_{\bullet}$ are complete flags of $\mathbb{C}^n$. Then,
\begin{center}
$D_X := \{C \in Hilb_{P_{d,m}(T)}(G(k,n)) \rvert C \cap X \neq 
\emptyset \}$

$D_Y := \{C \in Hilb_{P_{d,m}(T)}(G(k,n)) \rvert C \cap Y \neq 
\emptyset \}$.
\end{center}

\begin{theorem}
Let $d \geq 3, n \geq3, 1< k \leq \lfloor \frac{n}{2} \rfloor$, and $P_{d,m}(T) = \binom{T+m}{T} - \binom{T+m-d}{m}$ for all integer $2< m \leq k$. The Nef cone of the Hilbert scheme $Hilb_{P_{d,m}(T)}(G(k,n))$ is a cone spanned by the pull-backs of generators of the Nef cone of flag varieties, $[D_X]$, and $[D_Y]$.
\end{theorem}

In \S 2, we give preliminaries of the N\'{e}ron-Severi group and Nef cone. In \S 3, we suggest lemmas for finding generators of Nef cone and prove Theorem 1.2.

\section{Preliminaries}

 We recall Theorems 2.1 and 2.2 from Theorem 5.9 in \cite{1}.

\begin{theorem}
Let $F(a_1,\cdots,a_r;n)$ be a $r$-step flag variety. This flag variety admits $r$ canonical projections $\pi_i : F(a_1,\cdots,a_r;n) \to G(a_i,n)$ for each $i=1,\cdots,r$. Then, the N\'{e}ron-Severi group of the flag variety is generated by the pull-backs of the first Schubert classes of each Grassmannian $G(a_i,n)$.
\end{theorem}

\begin{theorem}
Let $\mathcal{E}$ be a vector bundle on a smooth variety $X$, and let $\pi : \mathbb{P} \mathcal{E} \to X$ be the projectivized bundle. If we denote the Picard number of $X$ by $Pic(X)$, then 
\begin{center}
$Pic(\mathbb{P}\mathcal{E}) = Pic(X)+1$
\end{center}
\end{theorem}

\begin{lemma}
Let $X$ be a projective variety with the picard number $n$. Let $D_1,\cdots,D_n$ be generators of N\'{e}ron-Severi group $NS(X)$ of $X$. If there are independent irreducible curves $\gamma_1,\cdots,\gamma_n \subset X$ such that
\begin{center}
$(D_i \cdot \gamma_j) = \delta_{ij}$ ~~~ for all $1 \leq i,j \leq n$,
\end{center}
then the Nef cone of $X$ is spanned by $D_1,\cdots,D_n$.
\end{lemma}
\begin{proof}
By \cite{2}, there is a vector space of numerical equivalence class of one-cycles denoted by $N_1(X)$. In this vector space, we can define the \textit{cone of curves} $NE(X)$ which is the cone spanned by effective one-cycles on $X$ and by taking its closure we can consider the \textit{closed cone of curves} $\overline{NE}(X)$. Since there exists a perfect pairing between $N_1(X)$ and the N\'{e}ron-Severi group $NS(X)$, this pairing induces the perfect pairing between the closed cone of curves and the Nef cone. Therefore, Nef cone of $X$ is a dual cone of the closed cone of curves. From the conditions above, the closed cone of curves is spanned by $\gamma_1,\cdots,\gamma_n$ and they act like duals of $D_1,\cdots,D_n$. As a result, we can conclude that classes of $D_1,\cdots,D_n$ span the Nef cone of $X$.
\end{proof}

\section{Nef cone of $Hilb_{P_{d,m}(T)}(G(k,n))$ for $m> 2,d\geq 3$}
In this section, I will always assume following conditions : \\
$\bullet$ $d \geq 3$\\
$\bullet$ $n \geq 4$, $1<k \leq \lfloor \frac{n}{2} \rfloor$ ($\because G(k,n) \cong G(n-k,n)$) \\
$\bullet$ Base field is $\mathbb{C}$ \\.
$\bullet$ $P_{d,m}(T) = \binom{T+m}{T} - \binom{T+m-d}{m}$ for all integer $2< m \leq k$\\

In \cite{3}, we show that the Hilbert scheme $Hilb_{P_{d,m}(T)}(G(k,n))$ is a projective bundle $\mathbb{P}(Sym^d \mathcal{S}^*)$ over a disjoint union of two flag varieties $F(k-1,k+m;n)$ and $F(k-m,k+1;n)$. 

\begin{center}
\begin{tikzpicture}[every node/.style={midway}]
        \matrix[column sep={14em,between origins}, row sep={5em}] at (0,0) {
\node(11) {$Hilb_{P_{d,m}(T)}(G(k,n))$};&\node(12) {$\mathbb{P}(Sym^d \mathcal{S}^*)$};&\node(13) {$ $};\\
\node(21) {$ $};&\node(22) {$F(k-1,k+m;n) \sqcup F(k-m,k+1;n)$};&\node(23) {\textit{space of $\mathbb{P}^m$}};\\
        };
        %\draw[right hook->] (11) -- (B1) node[left]{$j_A$};
        %\draw[left hook->] (13) -- (12) node[above]{$i_B$};
        %\draw[->] (22) -- (33) node[above]{$\pi_{a'}$};
        %\draw[dashed,->] (21) -- (22) node[below]{$\sigma_A$};
        \draw[->] (12) -- (22) ;
        \draw[<->] (11) -- (12) node[above]{$\cong$};
        \draw[<->] (22) -- (23) node[above]{$\cong$};
\end{tikzpicture}
\end{center}

The flag variety $F(k-1,k+m;n)$ parametrizes the space of $\mathbb{P}^m$ in the Grassmannian $G(k,n)$ a of the cohomology class $\sigma_{n-k,\cdots,n-k,n-k-m}$. The other flag variety $F(k-m,k+1;n)$ parametrizes the space of $\mathbb{P}^m$ of the cohomology class $\sigma_{n-k,\cdots,n-k,n-k-1,\cdots,n-k-1}$ ($m$ times $(n-k-1)$'s).

From now on, we will call a degree $d$ hypersurfaces in $\mathbb{P}^m$ as a degree $d$, $(m-1)$-fold in this paper and we will denote the unique projective space $\mathbb{P}^m$ that contains this $(m-1)$-fold by the \textit{span} of the degree $d$ $(m-1)$-fold.

Since $Hilb_{P_{d,m}(T)}(G(k,n))$ is the disjoint union of two projective bundles over flag varieties, if we denote the projective bundle over $F(k-1,k+m;n)$ by $\mathcal{M}_1$ and the projective bundle over $F(k-m,k+1;n)$ by $\mathcal{M}_2$, then we can claim that the Nef cone of the Hilbert scheme $Hilb_{P_{d,m}(T)}(G(k,n))$ is spanned by $Nef(\mathcal{M}_1)$ and $Nef(\mathcal{M}_2)$. Moreover, by the symmetric property of Schubert varieties, it is enough to compute the Nef cone one component of the Hilbert scheme $Hilb_{P_{d,m}(T)}(G(k,n))$.

\begin{remark}
When $k+m=n$, then $F(k-1,k+m;n) \cong G(k-1,n)$ and if $k-m=0$, then $F(k-m,k+1;n) \cong G(k+1,n)$. The proof for these cases are not different from the general case when $k+m<n$ and $k-m>0$. But, since we want to work on the general case first, we will always assume that $k+m<n$ and $k-m>0$ for the later subsections.\\
\end{remark}

By theorem 2.1 and 2.2, the Picard number of $\mathcal{M}_1$ is 3. Moreover, two generators of the N\'{e}ron-Severi group $NS(\mathcal{M}_1)$ can be determined by theorem 2.1. In this subsection, we will compute the N\'{e}ron-Severi group of $\mathcal{M}_1$ by observing the last generator, and compute the Nef cone of $\mathcal{M}_1$.

\begin{lemma}
Let $F_{\bullet}$ be a complete flag in $\mathbb{C}^n$ and $X := \Sigma_m(F_{\bullet}) \subset G(k,n)$ is a general special Schubert subvariety in $G(k,n)$. Then, a set $D_X:=\{C \in \mathcal{M}_1 \rvert C \cap X \neq \emptyset \}$ is a codimension 1 subvariety of $\mathcal{M}_1$.
\end{lemma}
\begin{proof}
Consider an incidence correspondence $I$ as follows :
\begin{center}
$I= \Big\{ (C, \widetilde{X},x,L) ~\Big\rvert~$ 
$\begin{array}{ll}
C \in \mathcal{M}_1,~ \widetilde{X} = \Sigma_m (\widetilde{F}_{\bullet})$ where $\widetilde{F}_{\bullet}$ is a complete flag of $\mathbb{C}^n \\
x \in C \cap \widetilde{X}$, $L \cong \mathbb{P}^m$ where $C \subset L \subset G(k,n)$
$\end{array} \Big\}$
\end{center}
Since we work on $d \geq 3$, every degree $d$, $(m-1)$-fold $C$ has a unique span $\mathbb{P}^m$ in $G(k,n)$ that contains $C$ \cite{3}. From this incidence correspondence we can define four canonical projections
\begin{center}
$\begin{array}{ll}
p_1: I \to \mathcal{M}_1 , & p_2 : I \to $(space of $\Sigma_m$ in $G(k,n)$)$\\
p_3 : I \to G(k,n) , & p_4 : I \to F(k-1,k+m;n)
\end{array}$.
\end{center}

Recall that $D_X = p_1 (p_2^{-1}(X))$. The cohomology class of $\widetilde{X}$ is $\sigma_m$ in $G(k,n)$ and this implies that the Schubert calculus intersection between general $\widetilde{X}$ and a general $L$ is 1 (\cite{1}).

So if we denote a general linear subspace $L_0$ which satisfies $\widetilde{X} \cap L_0$ is a point, and denote the other linear subspace as $L_1$, then we get the following picture :

Since there is no big difference between an arbitrary $\widetilde{X}$ and $X$ in $G(k,n)$, $p_2^{-1}(X)$ will give the same picture as above. Moreover, we get following results :\\

$\bullet$ $p_3 ( (p_2 \times p_4)^{-1} (X,L_0)) = X \cap L_0$ : a point (denote $X \cap L_0 = \{q\}$)\\

$\bullet$ $ (p_2 \times p_4)^{-1} (X,L_0) = \{(C,X,q,L_0) ~\rvert~ q \in C \subset L_0 \} \cong \mathbb{P}^{\binom{m+d}{d}-2}$ \\
$\implies$ the space of degree $d$ hypersurfaces in $L_0 \cong \mathbb{P}^m$ such that contains a point $q$\\

$\bullet$ $p_3 ( (p_2 \times p_4)^{-1} (X,L_1)) = X \cap L_1$ : a linear subspace of $L_1$ bigger than a point\\

$\bullet$ $p_1 ( (p_2 \times p_4)^{-1} (X,L_1)) \cong \mathbb{P}^{\binom{m+d}{d}-1}$ \\
$\implies$ the space of degree $d$ hypersurfaces in $L_1 \cong \mathbb{P}^m$ \\
($\because$ A line intersects with any hypersurface in $\mathbb{P}^m$ via Fundamental theorem of calculus).\\

Consider that the degree of hypersurface $d$ is greater than 2. Because of this, every hypersurface $C$ has a unique span $L \cong \mathbb{P}^m$ and this tells us that $D_X=p_1(p_2^{-1}(X)$ is isomorphic to $(p_1 \times p_4)(p_2^{-1}(X))$. In better words, if we define a map $\psi : \mathcal{M}_1 \to \mathcal{M}_1 \times F(k-1,k+m;n)$ as $C \mapsto (C,L)$ where $L$ is a $m$-dimensional projective space in $G(k,n)$ which $C$ is embedded into $L$, then $d \geq 3$ condition implies that $\psi$ is an isomorphism. And $(p_1 \times p_4)(p_2^{-1}(X)) = \psi \Big( p_1(p_2^{-1}(X) \Big)$. Using the map $\psi$, we can measure the codimension of $D_X$ in $\mathcal{M}_1$ by computing the codimension of $(p_1 \times p_4)(p_2^{-1}(X))$ in $\psi (\mathcal{M}_1)$.

With the picture of $p_2^{-1}(X)$ above, on a general $L_0$, a set of hypersurfaces of $(p_1 \times p_4)(p_2^{-1}(X))$ is a $\Big( \binom{m+d}{d}-2 \Big)$-dimensional projective space and $\psi (\mathcal{M}_1)$ is a $\Big( \binom{m+d}{d}-1 \Big)$-dimensional projective space. So it has codimension 1 for general $L_0$. However, on a special $L_1$, both $(p_1 \times p_4)(p_2^{-1}(X))$ and $\psi (\mathcal{M}_1)$ induce the same set of hypersurfaces, which is a $\Big( \binom{m+d}{d}-1 \Big)$-dimensional projective space. But, since the set of special $L_1$'s in $F(k-1,k+m;n)$ has at least codimension 1, we can also claim that the codimension of $(p_1 \times p_4)(p_2^{-1}(X))$ in $\psi (\mathcal{M}_1)$ restricted to special $L_1$ has codimension at least 1. Therefore, we get that the codimension of $(p_1 \times p_4)(p_2^{-1}(X))$ in $\psi (\mathcal{M}_1)$ is 1 and this completes the proof.
\end{proof}

With this lemma, we can claim that $D_X$ is a divisor in $\mathcal{M}_1$.

\begin{theorem}
$NS(\mathcal{M}_1)$ is generated by pull-back of the $NS(F(k-1,k+m;n))$ and a divisor class associated to $D_X := \{ C \in \mathcal{M}_1 \rvert C \cap X \neq \emptyset \}$ where $X := \Sigma_m(F_{\bullet}) \subset G(k,n)$ is a general special Schubert subvariety in $G(k,n)$.
\end{theorem}

\begin{proof}
By Theorems 2.1 and 2.2 above, the generators of $NS(F(k-1,k+m;n))$ induce generators of $NS(\mathcal{M}_1)$ and $rank~NS(\mathcal{M}_1)=3$. If we can show that a family of curves $D_X$ induces a divisor class which is independent to $NS(F(k-1,k+m;n))$, then we can claim that a divisor associated to $D_X$ is the other generator of $NS(\mathcal{M}_1)$.

\begin{center}
\begin{tikzpicture}[every node/.style={midway}]
        \matrix[column sep={12em,between origins}, row sep={3em}] at (0,0) {
\node(11) {$\mathcal{M}_1$};&\node(12) {$\mathbb{P}(Sym^d \mathcal{S}^*)$};&\node(13) {$ $};\\
\node(21) {$ $};&\node(22) {$F(k-1,k+m;n)$};&\node(23) {$ $};\\
\node(31) {$G(k-1,n)$};&\node(32) {$ $};&\node(33) {$G(k+m,n)$};\\
\node(41) {$\Sigma_1(F_{\bullet}^{(1,k-1)})$};&\node(42) {$ $};&\node(43) {$\Sigma_1(F_{\bullet}^{(1,k+m)})$};\\
        };
        %\draw[right hook->] (11) -- (B1) node[left]{$j_A$};
        %\draw[left hook->] (13) -- (12) node[above]{$i_B$};
        %\draw[->] (22) -- (33) node[above]{$\pi_{a'}$};
        %\draw[dashed,->] (21) -- (22) node[below]{$\sigma_A$};
        \draw[->] (12) -- (22) node[right]{$\varphi$};
        \draw[<->] (11) -- (12) node[above]{$\cong$};
        \draw[->] (22) -- (31) node[above]{$\pi_1$};
        \draw[->] (22) -- (33) node[above]{$\pi_2$};
        \draw[right hook->] (41) -- (31);
        \draw[right hook->] (43) -- (33);
\end{tikzpicture}
\end{center}

 $\mathcal{M}_1$ can be expressed as above. If we pick arbitrary complete flags $F_{\bullet}^{(1,k-1)}$ and $F_{\bullet}^{(1,k+m)}$ in $\mathcal{C}^n$, then the first Schubert classes in each Grassmannian give the generators of $NS(F(k-1,k+m;n))$. Denote as
\begin{center}
$c_1 := [\Sigma_1(F_{\bullet}^{(1,k-1)})]$ ~~~ and ~~~ $c_2 := [\Sigma_1(F_{\bullet}^{(1,k+m)})]$,

$NS(F(k-1,k+m;n)) = \mathbb{Z} \pi_1^* c_1 \oplus \mathbb{Z} \pi_2^* c_2$.
\end{center}

Since $\mathcal{M}_1$ is a projective bundle over $F(k-1,k+m;n)$, $(\pi_1 \circ \varphi)^*(c_1)$ and  $(\pi_2 \circ \varphi)^*(c_2)$ are linearly independent in $NS(\mathcal{M}_1)$ and we need to show that the class $[D_X]$ is linearly independent with respect to these two classes. To show this, in Lemma 3.4 we exhibit a curve $\gamma$ dual to $[D_X]$.
\end{proof}

\begin{corollary}
$NS(\mathcal{M}_2)$ is generated by pull-back of the $NS(F(k-m,k+1;n))$ and a divisor class associated to $D_Y := \{ C \in \mathcal{M}_2 \rvert C \cap Y \neq \emptyset \}$ where $Y := \Sigma_{1,\cdots,1}(F_{\bullet}) \subset G(k,n)$ is a general special Schubert subvariety in $G(k,n)$ ($codim (Y \subset G(k,n))=m$).
\end{corollary}

\begin{lemma}
There is a 1-dimensional family $\gamma$ of hypersurfaces in $\mathcal{M}_1$ such that satisfies following conditions :
\begin{center}
$(\pi_1 \circ \varphi)^*(c_1) \cdot \gamma =0$

$(\pi_2 \circ \varphi)^*(c_2) \cdot \gamma =0$

$[D_X]\cdot \gamma =1$
\end{center}
\end{lemma}

\begin{proof}
It is enough to find a curve $\gamma$ in $\mathcal{M}_1$  which intersect with $D_X$ at a point and disjoint with $\mathbb{P}(Sym^d \mathcal{S}^d)$ over $\pi_1^{-1}(\Sigma_1(F_{\bullet}^{(1,k-1)}))$ and $\pi_2^{-1}(\Sigma_1(F_{\bullet}^{(1,k+m)}))$.

Before constructing a curve $\gamma$, let's' take a general $m$-dimensional linear projective space in $G(k,n)$. Since we work in $\mathcal{M}_1$ such linear space corresponds to a pair $(V_1,V_2) \in F(k-1,k+m;n)$ via
\begin{center}
$\{ V\in G(k,n) \rvert V_1 \subseteq V \subseteq V_2 \} \cong \mathbb{P}^m$.
\end{center}
By generality, we can find a pair $(V_1,V_2)$ satisfying
\begin{center}
$V_1 \cap F^{(1,k-1)}_{n-k+1} = 0$ and $V_2 \cap F^{(1,k+m)}_{n-k-m}=0$

$\mathbb{P}(V_2/V_1) \cap X = \{W\} \subset G(k,n)$

(where $W$ is $k$-dimensional space in $\mathbb{C}^n$ such that $V_1 \subset W \subset V_2$, $W\in X$)
\end{center}

With this pair $(V_1,V_2)$ and a point $W \in G(k,n)$, we can express $W$ as a point in $\mathbb{P}^m$ by using local coordinates of $\mathbb{P}(V_2/V_1)$. With the same local coordinates, we can find two degree $d$ homogeneous polynomials $f,g \neq 0$ in $\mathbb{P}(V_2/V_1)$ such that $f(W)=0,g(W) \neq 0$, then each polynomial defines an element in $\mathcal{M}_1$. Now, let $\gamma$ be a pencil which is spanned by $f$ and $g$.
\begin{center}
$\gamma := \Big\{ C \in \mathcal{M}_1 ~\Big\rvert~ \begin{array}{l}
C \subset \mathbb{P}(V_2/V_1), \\
C \text{ is a zero locus of } sf+tg \text{ where } [s:t] \in \mathbb{P}^1
\end{array} \Big\}$
\end{center}

If we simply denote $\mathcal{N}_{1,\mathcal{M}_1}$ is $\mathcal{M}_1$ restricted to the base $\pi_1^{-1}(G(k-1,n))$ and $\mathcal{N}_{2,\mathcal{M}_1}$ is $\mathcal{M}_1$ restricted to the base $\pi_2^{-1}(G(k+m,n))$, then we show that
\begin{center}
$V_1 \cap F^{(1,k-1)}_{n-k+1} = 0 \implies \mathcal{N}_{1,\mathcal{M}_1} \cap \gamma = \emptyset \implies (\pi_1 \circ \varphi)^*(c_1) \cdot \gamma =0$

and

$V_2 \cap F^{(1,k+m)}_{n-k-m}=0 \implies \mathcal{N}_{2,\mathcal{M}_1} \cap \gamma = \emptyset \implies (\pi_2 \circ \varphi)^*(c_2) \cdot \gamma =0$.
\end{center}

Moreover, by the definition of $\gamma$, $D_X \cap \gamma$ is a singleton of degree $d$ $(m-1)$-fold contained in $\mathbb{P}(V_2/V_1)$ defined by the polynomial $f$ (which is the case $[s:t]=[1:0]$ in $\gamma$). Therefore, we have
\begin{center}
$[D_X]\cdot \gamma =1$
\end{center}
Therefore, $\gamma$ is a dual curve of the class $[D_X]$ in $NS(\mathcal{M}_1)$.
\end{proof}

\begin{remark}
In this paper, we will call this 1-dimensional family $\gamma$ as a \textit{dual curve} of the divisor $[D_X]$, because this curve acts like a dual element among the generators of the Nef cone of the Hilbert scheme.
\end{remark}

\begin{theorem}
The Nef cone $Nef(\mathcal{M}_1)$ is spanned by $(\pi_1 \circ \varphi)^*(c_1)$, $(\pi_2 \circ \varphi)^*(c_2)$, and $[D_X]$.
\end{theorem}

\begin{proof}
In Lemma 3.4, we find a curve $\gamma$ in $\mathcal{M}_1$ satisfying 
\begin{center}
$(\pi_1 \circ \varphi)^*(c_1) \cdot \gamma =0$

$(\pi_2 \circ \varphi)^*(c_2) \cdot \gamma =0$

$[D_X]\cdot \gamma =1$.
\end{center}
If we can find curves $\gamma'$ and $\gamma''$ in $\mathcal{M}_1$ satisfying

\begin{center}
\begin{table}[h]
\begin{tabular}{|c||c|c|c|}
\hline
$\cdot$ & $\gamma$ & $\gamma'$ & $\gamma''$ \\
\hline\hline
$[D_X]$ & 1 & 0 & 0 \\
\hline
$(\pi_1 \circ \varphi)^*(c_1)$ & 0 & 1 & 0 \\
\hline
$(\pi_2 \circ \varphi)^*(c_2)$ & 0 & 0 & 1 \\
\hline
\end{tabular}
\end{table}
\end{center}
then by using \textit{lemma 2.3}, we can complete the proof.\\

Two Schubert divisors in $F(k-1,k+m;n)$ are \\
(1) flags $(W_1 \subset W_2)$ where $W_1$ intersects a fixed $(n-k+1)$-dimensional subspace \\
(2) flags $(W_1 \subset W_2)$ where $W_2$ intersects a fixed $(n-k-m)$-dimensional subspace \\

In $F(k-1,k+m;n)$, the dual Schubert curves are given by \\
(1) flags $(W_1 \subset W_2)$ where $W_2$ is fixed and $W_1$ varies in a pencil contained in $W_2$\\
(2) flags $(W_1 \subset W_2)$ where $W_1$ is fixed and $W_2$ varies in a pencil whose base locus $B$ contains $W_1$ \\

Both $\gamma'$ and $\gamma''$ parametrize hypersurfaces of degree $d$ which are $d$-th powers of linear spaces in these pencils of $\mathbb{P}^m$ whose support does not change. In both dual Schubert curves (1) and (2), the $\mathbb{P}^m$'s have a common $\mathbb{P}^{m-1}$, because each curve is a family of $\mathbb{P}^m$'s rotating around a $\mathbb{P}^{m-1}$. Therefore, by taking $d$-th power of such $\mathbb{P}^{m-1}$, curves (1) and (2) induce $\gamma'$ and $\gamma''$ respectively and satisfy the condition above.

\end{proof}

\begin{corollary}
$Nef(\mathcal{M}_2)$ is generated by pull-back of the $NS(F(k-m,k+1;n))$ and $[D_Y]$.
\end{corollary}

\begin{remark}
If $k+m=n$, then $F(k-1,k+m;n) \cong G(k-1,n)$ and the rank of $NS(\mathcal{M}_1)=2$. In this case, there is no meaning for the projection map $\pi_2 : F(k-1,k+m;n) \to G(k+m,n)$ and we can't construct $\gamma''$. But, with the same method above, the Nef cone $Nef(\mathcal{M}_1)$ is a 2-dimensional cone spanned by $(\pi_1 \circ \varphi)^* (c_1)$ and $[D_X]$.
\end{remark}

The Hilbert scheme $Hilb_{P_{d,m}(T)}(G(k,n))$ is a disjoint union $\mathcal{M}_1$ and $\mathcal{M}_2$, and we computed the Nef cone of each components in theorem 3.6 and corollary 3.6.1. Therefore, if we have following diagram :
\begin{center}
\begin{tikzpicture}[every node/.style={midway}]
        \matrix[column sep={12em,between origins}, row sep={3em}] at (0,0) {
\node(11) {$\mathcal{M}_2$};&\node(12) {$\mathbb{P}(Sym^d \mathcal{S}^*)$};&\node(13) {$ $};\\
\node(21) {$ $};&\node(22) {$F(k-m,k+1;n)$};&\node(23) {$ $};\\
\node(31) {$G(k-m,n)$};&\node(32) {$ $};&\node(33) {$G(k+1,n)$};\\
\node(41) {$\Sigma_1(F_{\bullet}^{(2,k-m)})$};&\node(42) {$ $};&\node(43) {$\Sigma_1(F_{\bullet}^{(2,k+1)})$};\\
        };
        %\draw[right hook->] (11) -- (B1) node[left]{$j_A$};
        %\draw[left hook->] (13) -- (12) node[above]{$i_B$};
        %\draw[->] (22) -- (33) node[above]{$\pi_{a'}$};
        %\draw[dashed,->] (21) -- (22) node[below]{$\sigma_A$};
        \draw[->] (12) -- (22) node[right]{$\varphi'$};
        \draw[<->] (11) -- (12) node[above]{$\cong$};
        \draw[->] (22) -- (31) node[above]{$\pi'_1$};
        \draw[->] (22) -- (33) node[above]{$\pi'_2$};
        \draw[right hook->] (41) -- (31);
        \draw[right hook->] (43) -- (33);
\end{tikzpicture}

where $F_{\bullet}^{(2,k-m)}$ and $F_{\bullet}^{(2,k+1)}$ are complete flags in $\mathbb{C}^n$,
\end{center}
and if we denote 
\begin{center}
$c'_1 :=[\Sigma_1(F_{\bullet}^{(2,k-m)})]$ and $c'_2 :=[\Sigma_1(F_{\bullet}^{(2,k+1)})]$,
\end{center}
then we get the following result.

\begin{theorem}
The Nef cone of the Hilbert scheme $Hilb_{P_{d,m}(T)}(G(k,n))$ is a cone spanned by following classes : 

(i) $k+m<n$ and $k-m>0$ case : 6 classes \\
$\implies$ $(\pi_1 \circ \varphi)^*(c_1)$, $(\pi_2 \circ \varphi)^*(c_2)$,$[D_X]$, $(\pi'_1 \circ \varphi')^*(c'_1)$, $(\pi'_2 \circ \varphi')^*(c'_2)$, and $[D_Y]$

(ii) $k+m=n$ and $k-m>0$ case : 5 classes \\
$\implies$ $(\pi_1 \circ \varphi)^*(c_1)$, $(\pi'_1 \circ \varphi')^*(c'_1)$, $(\pi'_2 \circ \varphi')^*(c'_2)$, and $[D_Y]$

(iii) $k+m<n$ and $k-m=0$ case : 5 classes \\
$\implies$ $(\pi_1 \circ \varphi)^*(c_1)$, $(\pi_2 \circ \varphi)^*(c_2)$,$[D_X]$, $(\pi'_2 \circ \varphi')^*(c'_2)$, and $[D_Y]$

(iii) $k+m=n$ and $k-m=0$ case : 4 classes \\
$\implies$ $(\pi_1 \circ \varphi)^*(c_1)$,$[D_X]$, $(\pi'_2 \circ \varphi')^*(c'_2)$, and $[D_Y]$
\end{theorem}

\bibliographystyle{amsref}

\end{document}